\documentclass[a4paper,12pt,twoside]{amsart}
\usepackage[T1]{fontenc}
\usepackage{vmargin}
\usepackage{amssymb}
 \usepackage{csquotes }
\usepackage{amsmath}
\usepackage{mathrsfs}
\usepackage[dvips]{graphicx}
\usepackage{vmargin}
\usepackage{color}
\usepackage{amscd}
\DeclareGraphicsExtensions{.pdf, .png, .jpg, .mps}
\usepackage[colorlinks=false]{hyperref}
\definecolor{myc}{cmyk}{0.0009,0.8,0.8,0.00}
\newtheorem{theorem}{Theorem}
\newtheorem{lem}[theorem]{Lemma}
\newtheorem{pro}[theorem]{Proposition}

\newtheorem{cor}[theorem]{Corollary}
\newtheorem{oss}[theorem]{Remark}

\newtheorem*{nb}{\footnotesize {N.B}}

\def\p{\partial}

\def\io{{\infty}}

\def\dive{\operatorname{div}}
\def\moo{C^{\io}}
\def\mooc{C^{\io}_{\textit c}}

\def\R{\mathbb R}

\def\poscal#1#2{\langle#1,#2\rangle}

\def\norm#1{\Vert#1\Vert}
\def\val#1{\vert#1\vert}
\def\Val#1{\left\vert#1\right\vert}

\def\l2{L^2(\R^{n})}
\def\L2{L^2(\R^{2n})}

\def\lip{\operatorname{Lip}}

\def\vs{\vskip.3cm}

\def\mat22#1#2#3#4{\begin{pmatrix}#1&#2\\ #3&#4\end{pmatrix}}

\def\XXint#1#2#3{{\setbox0=\hbox{$#1{#2#3}{\int}$}
     \vcenter{\hbox{$#2#3$}}\kern-.5\wd0}}

\def\beq{\begin{equation}}
\def\eeq{\end{equation}}

\begin{document}
\title{Some natural subspaces 	and quotient spaces of $L^1$}
\author{Gilles Godefroy and Nicolas Lerner}
\address{\noindent \textsc{G. Godefroy and N. Lerner, Institut de Math\'ematiques de Jussieu,
Universit\'e Pierre et Marie Curie (Paris VI),
4 Place Jussieu,
75252 Paris cedex 05,
France}}
\email{gilles.godefroy@imj-prg.fr ; nicolas.lerner@imj-prg.fr}
\date{\today}
\begin{abstract}
We show that the space $\text{Lip}_0(\mathbb R^n)$ is the dual space of
$L^{1}({\mathbb R}^{n}; {\mathbb R}^{n})/N$ where $N$ is the subspace of $L^{1}({\mathbb R}^{n}; {\mathbb R}^{n})$
consisting of vector fields whose divergence {vanishes}.
We prove that although the quotient space $L^{1}({\mathbb R}^{n}; {\mathbb R}^{n})/N$ is weakly sequentially
complete, the subspace $N$ is not nicely placed - in other words, its unit ball is not closed
for the topology $\tau_m$ of local convergence in measure. We prove that if $\Omega$ is a bounded
open star-shaped subset of $\mathbb {R}^n$ and $X$ is a closed subspace of $L^1(\Omega)$
consisting of continuous functions, then the unit ball of $X$ is compact for the compact-open topology
on $\Omega$. It follows in particular that such spaces $X$, when they have Grothendieck's approximation property,
have unconditional finite-dimensional decompositions and are isomorphic to weak*-closed subspaces of $l^1$.
Numerous examples are provided where such results apply.
\end{abstract}
\maketitle
\noindent
\section{Introduction}
Among the wealth of important discoveries due to Uffe Haagerup, one can single out what is now universally called Haagerup's approximation property,
a fundamental concept in operator algebras and their various applications. The present work investigates approximation properties on a much lesser scale, 
and the tools we use are familiar to every functional analyst: among them, dilation operators on star-shaped domains and Grothendieck's approximation property.
Our purpose is to analyse some natural subspaces (and quotient spaces) of $L^1$. We are therefore outside the reflexive world, where the lack of compactness
can hurt some proofs and where some natural operators become unbounded. This leads us to weaken the topologies, thus to enter the realm of non-locally convex spaces 
and to use the topology $\tau_m$ of convergence in measure. Such tools will allow us to provide satisfactory results on subspaces of $L^1$ which satisfy
quite weak assumptions: for instance, we show (Corollary 10) that if $\Omega$ is a star-shaped bounded open subset of $\mathbb{R}^n$, if $X$ is a closed subspace of
$L^1(\Omega)$ consisting of continuous functions and stable under the dilation operators  $(T_\rho)$, and if $X$ has Grothendieck's approximation property, 
then $X$ is isomorphic to a weak-star closed subspace of $l^1$. Hence such a space has a
``somewhat discrete''
structure. It turns out that these assumptions are satisfied by 
many classical spaces. Moreover these spaces $X$ have unconditional finite dimensional decompositions. We therefore apply a rule of thumb which has been discovered by Nigel Kalton and some of his co-authors: homogeneity of a Banach space $X$ implies unconditionality on $X$.
\vskip 3 mm
We now outline the content of this note. Let $\Omega$ be an open subset of $\mathbb{R}^n$, equipped with the Lebesgue measure denoted $m$. A closed subspace $X$ of $L^1(\Omega)$ is called nicely placed
if its unit ball is closed for the topology $\tau_m$ of local convergence in measure (see Chapter IV in \cite{HWW} or \cite{JaLe}). It is known that the quotient space $L^1/X$ is $L$-complemented in its bidual (and thus weakly sequentially complete) when $X$ is nicely placed, and the same conclusion holds when we consider integrable functions with values in a finite-dimensional normed space (see e.g. p.200 in \cite{HWW}), in  particular integrable vector fields on $\mathbb{R}^n$. Our first result is somewhat negative: we show that the free space $\mathcal F(\mathbb R^n)$ over $\mathbb{R}^n$ is isometric to the quotient of the space $(L^{1}({\mathbb R}^{n}))^{n}=L^{1}({\mathbb R}^{n}; {\mathbb R}^{n})$ of integrable vector fields on $\mathbb{R}^n$ by the space $N$ of divergence-free vector fields, and we show that although $\mathcal F(\mathbb R^n)$ shares many properties of spaces which are $L$-complemented in their bidual, the space $N$ is {\sl not} nicely placed. This discards a natural conjecture, but leads to several questions. In the second (independent) part of our paper, we show that {if} $\Omega$ is star-shaped and bounded, homogeneous subspaces $X$ of $L^1(\Omega)$ consisting of continuous functions are very special examples of nicely placed subspaces: their unit ball $B_X$ is actually 
$\tau_m$-compact locally convex. This bears strong consequences on the structure of such spaces.
\section{Divergence-free vector fields and the space \texorpdfstring{$\mathcal F(\mathbb R^n)$}{FRn}}
We first provide a representation result for the predual of the space of Lipschitz functions on the space $\mathbb R^n$. 
We recall the usual notation
\begin{equation}\label{lipzero}
\text{Lip}_0(\mathbb R^n)=\{f:\mathbb R^n\rightarrow \mathbb R, \text{such that $f(0)=0$ and }
\sup_{x\not=y}\frac{\vert f(x)-f(y)\vert}{\vert x-y\vert}<+\infty.
\}
\end{equation}
Functions in $\text{Lip}_0(\mathbb R^n)$ are the continuous functions from $\R^{n}$ into $\R$ such that $f(0)=0$ 
with a distribution gradient in $L^{\io}(\R^{n})$.
The vector space $\text{Lip}_0(\mathbb R^n)$ is a Banach space, 
with norm
$$
\norm{f}_{\text{Lip}_{0}(\R^{n})}=\sup_{x\not=y}\frac{\vert f(x)-f(y)\vert}{\vert x-y\vert}=\norm{\nabla f}_{L^{\io}(\R^{n})}.
$$
$\text{Lip}_0(\mathbb R^n)$ is the dual space of the Banach space denoted $\mathcal F(\mathbb R^n)$,
also known as the Lipschitz-free space over $\mathbb R^n$. Let us recall that it has been recently shown by N. Weaver (\cite{Wea})
that the free space over an arbitrary metric space $M$ is strongly unique isometric predual of its dual  $\text{Lip}_0(M)$.
In particular, any Banach space whose dual is isometric to $\text{Lip}_0(\mathbb R^n)$ coincide with $\mathcal F(\mathbb R^n)$.
Following \cite{Ler}, we represent the space $\text{Lip}_0(\mathbb R^n)$ as a 
{closed subspace of $(L^{\io}(\R^{n}))^{n}=L^{\io}(\R^{n};\R^{n})$ (in fact the closed $L^{\io}$ currents),}
and then we check that this closed subspace is exactly the orthogonal space to a subspace $N$ of the predual $L^{1}({\mathbb R}^{n}; {\mathbb R}^{n})$.
Thus the free space is identical with the quotient space $L^{1}({\mathbb R}^{n}; {\mathbb R}^{n})/ N$.
This approach relies on de Rham's theorem on closed currents and an integration by parts. However, some technicalities are needed since derivatives must be taken
in the distribution sense (see Remark 6). It should be noted that free spaces over convex open subsets of $\mathbb R^n$ are similarly represented in the recent work
\cite{CKK}, without using Weaver's work - which requests  a slightly different approach.
\par
We begin with two simple lemmas.
\begin{lem}\label{l11}
 Let $X=L^{1}({\mathbb R}^{n}; {\mathbb R}^{n})$ be the Banach space of integrable vector fields
and let $N$ be the subspace of $X$ made of vector fields with null distribution divergence:
\begin{equation}\label{22}
N=\{(f_{j})_{1\le j\le n}\in X, \ \sum_{1\le j\le n}\frac{\partial  f_{j}}{\partial x_{j}}=0\}.
\end{equation}
Then $N$ is a closed subspace of $X$.
\end{lem}
\begin{nb}
 It is convenient to note the elements $F=(f_{j})_{1\le j\le n}\in L^{1}({\mathbb R}^{n}; {\mathbb R}^{n})$ as vector fields
 $$
 F=\sum_{1\le j\le n} f_{j}\frac{\p}{\p x_{j}}.
 $$
 The distribution divergence of $F$ is then defined by $\dive F=\sum_{1\le j\le n}\frac{\partial  f_{j}}{\partial x_{j}}.$
\end{nb}
\begin{proof}
 Let $\bigl(F_{k}=\sum_{1\le j\le n} f_{k,j}\p_{j}\bigr)_{k\ge 1}$ be a sequence of vector fields of $N$, converging in $X$ with limit $F=\sum_{1\le j\le n} f_{j}\p_{j}$
 (this means that for all $j\in \{1,\dots, n\}$, $\lim_{k}f_{j,k}=f_{j}$ in $L^{1}(\R^{n})$).
 Let $\phi\in\mooc(\R^{n})$: we have
 $$\poscal{\frac{\p f_{k,j}}{\p{x_{j}}}}{\phi}_{\mathscr D'(\R^{n}), \mathscr D(\R^{n})}=-
 \poscal{f_{k,j}}{\frac{\p\phi}{\p x_{j}}}_{\mathscr D'(\R^{n}), \mathscr D(\R^{n})}=-\int_{\R^{n}}
 f_{k,j}(x)\frac{\p\phi}{\p x_{j}}(x)dx,
 $$
 and consequently
 $$
 \lim_{k}\poscal{\frac{\p f_{k,j}}{\p{x_{j}}}}{\phi}_{\mathscr D'(\R^{n}), \mathscr D(\R^{n})}
 =-\int_{\R^{n}}f_{j}(x)\frac{\p\phi}{\p x_{j}}(x)dx=\poscal{\frac{\p f_{j}}{\p{x_{j}}}}{\phi}_{\mathscr D'(\R^{n}), \mathscr D(\R^{n})},
 $$
 which implies
 $
 0= \lim_{k}\poscal{\underbrace{\sum_{1\le j\le n}\frac{\p f_{k,j}}{\p{x_{j}}}}_{=0}}{\phi}_{\mathscr D'(\R^{n}), \mathscr D(\R^{n})}
 =\poscal{\sum_{1\le j\le n}\frac{\p f_{j}}{\p{x_{j}}}}{\phi}_{\mathscr D'(\R^{n}), \mathscr D(\R^{n})},
 $
 and thus
 $
 \dive F=0,
 $
  proving the sought result.
\end{proof}
\begin{lem}\label{lem2}
 The space $\lip_{0}(\R^{n})$ is isomorphic to the closed $L^{\io}(\R^{n})$ currents,
 i.e.
 to the subspace 
 \begin{equation}\label{currents}
\mathcal C_n=\{(u_{j})_{1\le j\le n}\in (L^{\io}(\R^{n}))^{n},\text{ such that }
 \frac{\p u_{j}}{\p x_{k}}= \frac{\p u_{k}}{\p x_{j}}\quad \text{for  $1\le j<k\le n$}
 \}.
\end{equation}
 More precisely, the mapping
 $$
 \lip_{0}(\R^{n})\ni a\mapsto da\in\mathcal C_n,
 $$
is an isomorphism of Banach spaces.
\end{lem}
\begin{nb}
 As in Lemma \ref{l11}, we can prove that $\mathcal C_n$ is a closed subspace of the Banach space 
 $(L^{\io}(\R^{n}))^{n}$.
 All the derivatives are taken in the distribution sense.
It is convenient to note the elements of $\mathcal C_n$ as
$u=\sum_{1\le j\le n} u_{j}dx_{j}$, so that for $a\in \lip_{0}(\R^{n})$,
we have
$$
da=\sum_{1\le j\le n}\frac{\p a}{\p x_{j}}dx_{j}.
$$ 
\end{nb}
\begin{proof}
From the definition of  $\lip_{0}(\R^{n})$, we see that $da$ is a $L^{\io}(\R^{n})$ current and also that $da$ is closed since, in the distribution sense, we have 
$$
\frac{\p^{2} a}{\p x_{j}\p x_{k}}=\frac{\p^{2} a}{\p x_{k}\p x_{j}},
$$ 
 meaning that the linear mapping given in the lemma is well-defined from 
 $\lip_{0}(\R^{n})$ into $ \mathcal C_n$.
 This mapping is also  isometric (and thus one-to-one)
 since 
 $$
 \norm{da}_{\mathcal C_n}=\norm{\nabla a}_{L^{\io}(\R^{n})}=\norm{a}_{\lip_{0}(\R^{n})}.
 $$
For concluding the proof, we need only to  prove that this mapping is onto:
in fact thanks to de Rham's theorem on closed currents
(see \cite{MR0346830} or \cite{MR0205028}), if $u\in \mathcal C_n$, there exists a distribution $w$ on $\R^{n}$ such that
$$
dw=u.
$$ 
As a result, the distribution $w$ has a gradient in $L^{p}_{loc}$ for any $p\in(1,+\io)$
and the Sobolev embedding theorem implies that
(taking $p>n$)   $w$ is a (H\"older)
continuous function. We can take now
$$
a(x) =w(x)-w(0),
$$
and we find that $a$ belongs to  $\lip_{0}(\R^{n})$ and satisfies
$da=u$.
  \end{proof}
 We now state and prove a representation result for $\mathcal F(\mathbb R^n)$.
\begin{pro}\label{5599}
Let $X=L^{1}({\mathbb R}^{n}; {\mathbb R}^{n})$ be the Banach space of integrable vector fields
and let $N$ be the closed subspace of $X$ made of vector fields with null distribution divergence as defined by \eqref{22}.
 Then, the free space $\mathcal F(\mathbb R^n)$ over $\mathbb R^n$ is isometric to $X/N$ and we have 
 $$
 \text{Lip}_0(\mathbb R^n)=(X/N)^{*}.
 $$
\end{pro}
\begin{proof}
Note that it suffices to prove the last equation $\text{Lip}_0(\mathbb R^n)=(X/N)^{*}$ since by Weaver's result
the isometric predual is unique.
 The case $n=1$ is easy since, in that case $N=\{0\}$, so that $X/N=L^{1}(\R)$; thanks to Lemma \ref{lem2},
 we have also
 $
 \text{Lip}_0(\mathbb R)=\mathcal C_{1}=L^{\io}(\R),
 $
 proving our claim which reduces to $\bigl(L^{1}(\R)\bigr)^{*}=L^{\io}(\R)$.
Let us now assume that $n\ge 2$. 
We start with a lemma.
\begin{lem}\label{lem4}
 We define
 $$
X\times \lip_{0}(\R^{n})\ni (f,a)\mapsto \Phi(f,a)=\int_{\R^{n}} \sum_{1\le j\le n} f_{j}\frac{\p a}{\p x_{j}}dx\in \R.
$$
The mapping $\Phi$ is bilinear continuous.
Moreover
for $a\in \lip_{0}(\R^{n})$ and   $f\in N$ (given by \eqref{22}),we have  $\Phi(f,a)=0$.
\end{lem}
\begin{proof}[Proof of the lemma]
The bilinearity and continuity of $\Phi$
are obvious.
Let $\rho\in \mooc(\R^{n};\R_{+})$, supported in the unit ball, even with integral 1; we set  for $\epsilon>0$$,
\rho_{\epsilon}(x)=\epsilon^{-n}\rho(x/\epsilon)$ and we 
define
$$
a_{\epsilon}(x)=(a\ast \rho_{\epsilon})(x)=\int a(y) \rho_{\epsilon}(x-y) dy.
$$
We note that $a\in \moo$ and  $da_{\epsilon}=da\ast \rho_{\epsilon}$, which is thus bounded in  $(L^{\io}(\R^{n}))^{n}$
by  $\norm{da}_{L^{\io}(\R^{n})}$ and 
converges a.e. towards  $da$,
thanks to Lebesgue differentiation Theorem\footnote{For $u\in L^{\io}(\R^{n})$,
we have
$\val{(u\ast \rho_{\epsilon})(x)-u(x)}=$
$$
\hskip55pt=\Val{\int\bigl(u(y)-u(x)\bigr)\rho_{\epsilon}(x-y) dy}
\le\underbrace{\frac{1}{\epsilon^{n}\val{\mathbb B^{n}}} \int_{\val{y-x}\le \epsilon}\val{u(y)-u(x)} dy}_{
\substack{\rightarrow 0,\ \text{a.e. in $x$}\\\text{(Lebesgue's differentiation theorem)}}}
\val{\mathbb B^{n}}\norm{\rho}_{L^{\io}(\R^{n})}.
$$}
(of course, no convergence in $L^{\io}$ is expected).
We have thus 
\begin{multline}\label{963}
\int_{\R^{n}} \sum_{1\le j\le n} f_{j}\frac{\p a}{\p x_{j}}dx=\lim_{\epsilon}
\int_{\R^{n}} \sum_{1\le j\le n} f_{j}\frac{\p a_{\epsilon}}{\p x_{j}}dx\\
=\lim_{\epsilon}\Bigl(\lim_{k\rightarrow+\io}
\int_{\R^{n}} \sum_{1\le j\le n} f_{j}(x)\frac{\p a_{\epsilon}}{\p x_{j}}(x) \chi_{0}(x/k)dx
\Bigr),
\end{multline}
where  $\chi_{0}$  is a  $\mooc$ function, valued in $[0,1]$, equal to  1 on $\val x\le 1/2$ and supported in $\val x\le 1$.
We note that 
\begin{multline}\label{964}
\int f_{j}(x)\frac{\p a_{\epsilon}}{\p x_{j}}(x) \chi_{0}(x/k)dx
=\poscal{f_{j}(x)}{\chi_{0}(x/k) \frac{\p a_{\epsilon}}{\p x_{j}}(x) }_{\mathscr D',\mathscr D}\\
=\poscal{f_{j}(x)}{ \frac{\p}{\p x_{j}}
\bigl\{a_{\epsilon}(x)\chi_{0}(x/k) \bigr\}}_{\mathscr D',\mathscr D}
-\poscal{f_{j}(x)}{ 
a_{\epsilon}(x)(\p_{j}\chi_{0})(x/k)k^{-1}
}_{\mathscr D',\mathscr D}
\\
=-\poscal{\frac{\p f_{j}}{\p x_{j}}(x)}{ 
a_{\epsilon}(x)\chi_{0}(x/k)}_{\mathscr D',\mathscr D}
-\int a_{\epsilon}(x)
f_{j}(x)(\p_{j}\chi_{0})(x/k)k^{-1} dx,
\end{multline}
and since $f\in N$,
we find
\begin{align}\label{aazz}
&\int_{\R^{n}} \sum_{1\le j\le n} f_{j}\frac{\p a}{\p x_{j}}dx=
-\lim_{\epsilon}\Bigl(\lim_{k\rightarrow+\io}
\int_{\R^{n}} 
a_{\epsilon}(x)
\bigl(\sum_{1\le j\le n}
f_{j}(x)(\p_{j}\chi_{0})(x/k)k^{-1}\bigr)
dx
\Bigr).
\end{align}
On the other hand,
the term $(\p_{j}\chi_{0})(x/k)$ is vanishing outside of $\{x, k/2<\val x< k\}$,
so that
\begin{multline}\label{7788}
\Val{\int_{\R^{n}}
a_{\epsilon}(x)
f_{j}(x)(\p_{j}\chi_{0})(x/k)k^{-1} dx}
\\\le 
\int_{\frac k2\le \val x\le k}
\val{a_{\epsilon}(x)-a_{\epsilon}(0)}
\val{f_{j}(x)}dx\ k^{-1} \norm{\p_{j}\chi_{0}}_{L^{\io}(\R^{n})}
\\
+\int_{\R^{n}}
\val{a_{\epsilon}(0)}
\val{f_{j}(x)}dx\ k^{-1} \norm{\p_{j}\chi_{0}}_{L^{\io}(\R^{n})}.
\end{multline}
Since $\norm{da_{\epsilon}}_{L^{\io}(\R^{n})}\le \norm{da}_{L^{\io}(\R^{n})}=L<+\io$, we obtain
$$
\val{a_{\epsilon}(x)-a_{\epsilon}(0)}\le L\val x,
$$
so  that the first term in the right-hand side of \eqref{7788} is bounded above by
$$
\int_{\frac k2\le \val x\le k}
L\val x
\val{f_{j}(x)}dx\ k^{-1} \norm{\p_{j}\chi_{0}}_{L^{\io}(\R^{n})}\le \int_{\val x\ge k/2}\val{f_{j}(x)}
dx
L\norm{\p_{j}\chi_{0}}_{L^{\io}(\R^{n})},
$$
which is independent of $\epsilon$ and goes to 0 when $k$ goes to $+\io$ since each $f_{j}$ belongs to $L^{1}(\R^{n})$.
Moreover, 
we have $a(0)=0$ and thus 
$$a_{\epsilon}(0)=\int \bigr (a(y)-a(0)\bigl)\rho(-y/\epsilon)\epsilon^{-n}dy,$$
so that 
 $$
 \val{a_{\epsilon}(0)}\le L\int \val y \rho(y/\epsilon)\epsilon^{-n}dy\le \epsilon C_{0}, \quad C_{0}=\int \val z \rho(z) dz,
 $$
and we obtain that 
 the second term in the right-hand side of \eqref{7788} is bounded above by
 $$
 \int_{\R^{n}}
\epsilon C_{0}
\val{f_{j}(x)}dx\ k^{-1} \norm{\p_{j}\chi_{0}}_{L^{\io}(\R^{n})},
 $$
 which goes to 0 when $k\rightarrow+\io$
  since each $f_{j}$ belongs to $L^{1}(\R^{n})$.
  Finally the right-hand side of \eqref{7788} goes to 0 when $k\rightarrow+\io$
  and this implies that the left-hand side of \eqref{aazz} is zero, which is the sought result.
This lemma implies that the mapping $\tilde\Phi$ defined on $X/N\times 
 \lip_{0}(\R^{n})$ 
 by
$$
\tilde\Phi\bigl(p(f),a\bigr)= \Phi(f,a),
$$
where $p:X\rightarrow X/N$ is the canonical surjection,
is well-defined and is a continuous bilinear mapping.
\end{proof}
$\bullet$ Going back to the proof of Theorem \ref{5599}, we see that $\tilde\Phi$ induces a continuous linear mapping
$\mathcal L$
from 
$ \lip_{0}(\R^{n})$
into 
$
(X/N)^{*}
$
defined
by
$$\lip_{0}(\R^{n})\ni a\mapsto \mathcal L(a)\in(X/N)^{*},\quad
(\mathcal L(a))(p(f))=\Phi(f,a).
$$
We check first that $\mathcal L$ is one-to-one.
\begin{lem}
 Let $a\in \lip_{0}(\R^{n})$ such that for all $f\in X$, $\Phi(f,a)=0$.
 Then we have $a=0$.
\end{lem}
\begin{proof}[Proof of the lemma]
Let $\chi_{k}\in \mooc(\R^{n}; \R_{+}), \chi_{k}=1$ on $\val x\le k$.
We have for $f=(\chi_{k} \frac{\p a}{\p x_{j}})_{1\le j\le n}$ (which belongs to $X$)
$$0=\Phi(f,a)=
\int_{\R^{n}}\chi_{k}(x)\sum_{1\le j\le n} \bigl(\frac{\p a}{\p x_{j}}(x)\bigr)^{2} dx,
$$
which implies that $da=0$ on $\val x\le k$ for any $k$ and thus $da=0$, inducing $a=0$.
\end{proof}
$\bullet$ Finally, let us prove that $\mathcal L$ is onto.
Let $\xi\in (X/N)^{*}$; since $\xi\circ p\in X^{*}=\bigl(L^{\io}(\R^{n})\bigr)^{n}$,
we find
$(u_{j})_{1\le j\le n}\in \bigl(L^{\io}(\R^{n})\bigr)^{n}$ such that 
$$
\poscal{\xi}{p(f)}_{(X/N)^{*},X/N}=\int \sum_{1\le j\le n}u_{j} f_{j} dx,\qquad \forall f\in N,\
 \int \sum_{1\le j\le n} f_{j}u_{j} dx=0.
$$
 Let $j, k$ be given in $\{1,\dots, n\}$. We have 
 $$
 \poscal{\frac{\p u_{j}}{\p x_{k}}-\frac{\p u_{k}}{\p x_{j}}}{\varphi}_{\mathscr D',\mathscr D}=
 \poscal{u_{j}}{-\frac{\p \varphi}{\p x_{k}}}_{\mathscr D',\mathscr D}
 + \poscal{u_{k}}{\frac{\p \varphi}{\p x_{j}}}_{\mathscr D',\mathscr D}
 =\int\bigl(u_{k}f_{k}+u_{j}f_{j}\bigr) dx,
 $$
with 
 $
 f_{j}=-\frac{\p \varphi}{\p x_{k}},\quad f_{k}=\frac{\p \varphi}{\p x_{j}}.
 $
 The vector field   $(f_{j}\p_{j}+f_{k}\p_{k})$ is  $L^{1}$  with null divergence and we get 
 \[\frac{\p u_{j}}{\p x_{k}}-\frac{\p u_{k}}{\p x_{j}}=0,
 \]
 proving that
the current
$
u=\sum_{1\le j\le n} u_{j}dx_{j}
$
is closed and thus belongs to $\mathcal C_{n}$ (see \eqref{currents}).
Lemma \ref{lem2}
implies that there exists $a\in \lip_{0}(\R^{n})$ such that $da=u$,
proving that $\mathcal L$ is onto.
\end{proof}
\begin{oss}\rm 
 The proof of Lemma \ref{lem4} is giving a little bit more than the statement of this lemma: in fact
 Formula \eqref{964} holds without the assumption $f\in N$ and we obtain from the sequel of the proof that, for $(f,a)\in X\times  \lip_{0}(\R^{n})$, and $\rho, \chi_{0}$ as in Lemma \ref{lem4},
 $$
 \int_{\R^{n}} \sum_{1\le j\le n} f_{j}\frac{\p a}{\p x_{j}}dx=
 -\lim_{\epsilon\rightarrow 0}\Bigl(
 \lim_{k\rightarrow +\io}
 \poscal{\sum_{1\le j\le n}\frac{\p  f_{j}}{\p x_{j}}}{(a\ast \rho_{\epsilon})(x) \chi_{0}(x/k)}_{\mathscr D',\mathscr D}\Bigr),
 $$
 a formula which can be written for the $L^{1}(\R^{n})$ vector field $F=\sum_{1\le j\le n}f_{j}\p_{x_{j}}$
 as
 \begin{equation}\label{last}
\int_{\R^{n}}F(a) dx
=
 -\lim_{\epsilon\rightarrow 0}\Bigl(
 \lim_{k\rightarrow +\io}
 \poscal{\dive F}{(a\ast \rho_{\epsilon})(x) \chi_{0}(x/k)}_{\mathscr D',\mathscr D}\Bigr).
\end{equation}
When $a$ belongs to $C^{1}_{c}(\R^{n})$, the above formula follows from a standard integration by parts and the right-hand side of \eqref{last} is 
$
 -\poscal{\dive F}{a}_{\mathscr D'^{{(1)}},C^{1}_{c}}
 $,
although in the more general case tackled here, we have to pay attention to the fact that $\dive F$ could be a distribution of order 1
which is not defined a priori on Lipschitz continuous functions.
\end{oss}
Note that for $n=2$, we have
$$
\mathcal F(\R^{2})= L^{1}(\R^{2};\R^{2})/\bigl(\nabla^{\perp}L^{2}(\R^{2})\cap L^{1}(\R^{2};\R^{2})\bigr),
$$
where $\nabla^{\perp}$ denotes the orthogonal gradient
{defined by $\nabla^{\perp}\psi=
(\p_{x_{2}}\psi,-\p_{x_{1}}\psi)$}, and for $n=3$,
$$
\mathcal F(\R^{3})= L^{1}(\R^{3};\R^{3})/\bigl(
\text{curl } L^{3/2}(\R^{3};\R^{3})
\cap 
L^{1}(\R^{3};\R^{3})
\bigr).
$$
\vs
As shown in \cite{CDW}, the Lipschitz-free space $\text{Lip}_0(\mathbb R^n)$ over $\mathbb R^n$ is weakly
sequentially complete. Since it is now represented as a quotient space of $L^1$, it is natural to wonder whether
the kernel $N$ of the quotient map is ``nicely placed" (see \cite{GKL}), in other words if its unit ball is closed in $L^1$
for the topology $\tau_m$ of local convergence in measure. Indeed, the quotient of $L^1$ by any such space enjoys a strong form
of weak sequential completeness (\cite{Pfitzner}). But one has:
\begin{pro}\label{p7}
 Let $n>1$, let $X=L^{1}({\mathbb R}^{n}; {\mathbb R}^{n})$ be the Banach space of integrable vector fields
and let $N$ be the subspace of $X$ of vector fields with null distribution divergence:
$$
N=\{(f_{j})_{1\le j\le n}\in X, \ \sum_{1\le j\le n}\frac{\partial  f_{j}}{\partial x_{j}}=0\}.
$$
Then $N$ is not nicely placed, that is, its unit ball is not closed for the topology $\tau_m$ of local convergence in measure.
\end{pro}
\begin{proof}
First observe that the space $N$ is translation invariant, and thus is stable under convolution with
integrable functions. If $N$ is nicely placed, it follows from Bocl\'e's differentiation lemma (\cite{Bo})
 that if a measure-valued vector field $X\in(\mathcal{M}({\mathbb R^n}))^n$ is divergence-free, then its absolutely continuous part is
 divergence-free as well. Indeed, Boclé's Lemma shows that if $(c_k)$ is an approximation of identity in the convolution algebra $L^1({\mathbb R^n})$
 and $\mu$ is a singular measure, then $(\mu*c_k)$ converges to $0$ in quasi-norm $\Vert~.~\Vert_p$ for all $0<p<1$, and it follows that a nicely placed
 translation-invariant space
 of measures is stable under the Radon-Nikodym projection (see the proof of Lemma 1.5 in \cite{Riesz}).
Let us provide an example of an unstable divergence-free vector field in the case $n=2$.
\par
Let $\chi\in C^{1}_{c}(\R^{2})$ be arbitrary, and $H=\mathbf 1_{\R_{+}}$. We consider the function $\psi\in L^{2}(\R^{2})$ defined by
$$
\psi(x_{1},x_{2})=\chi(x_{1},x_{2}) H(x_{1})
$$
The field 
$$
X=\nabla^{\perp}\psi=\frac{\p \psi}{\p x_{2}}\frac{\p}{\p x_{1}}-\frac{\p \psi}{\p x_{1}}\frac{\p}{\p x_{2}},
$$
is divergence-free. Moreover
\begin{multline*}
X=\frac{\p \chi}{\p x_{2}} H(x_{1})\frac{\p}{\p x_{1}}-\Bigl(
\frac{\p \chi}{\p x_{1}} H(x_{1})+\chi \delta_{0}(x_{1})
\Bigr)\frac{\p}{\p x_{2}}
\\=
\underbrace{\frac{\p \chi}{\p x_{2}} H(x_{1})\frac{\p}{\p x_{1}}-
\frac{\p \chi}{\p x_{1}} H(x_{1})\frac{\p}{\p x_{2}}}_{V}
\underbrace{-\chi \delta_{0}(x_{1})\frac{\p}{\p x_{2}}}_{W},
\end{multline*}
The field $V$ takes its values into  $L^{1}$, the field $W$ is singular and
$$
\text{div} (V+W)=0,\qquad
\text{div } W=-\frac{\p \chi}{\p x_{2}}\delta_{0}(x_{1})\not=0,
$$
as soon as $\frac{\p \chi}{\p x_{2}}(0,x_{2})$ does not vanish identically. This concludes the proof, and actually shows that the vector field $V\not\in N$
is the limit of the $\Vert~.~\Vert_1$-bounded sequence  $(X*c_k)\subset N$ for the topology of local convergence in measure.
\end{proof}
\vskip 3 mm
We recall that a Banach space $Z$ has property $(X)$ if every $z^{**}\in Z^{**}$ such that $z^{**}(weak^*-\sum x^*_k)=\sum z^{**}(x^*_k)$ for every weakly
unconditionally convergent series $(x^*_k)$ actually belongs to $Z$ (see p. 147 in  \cite{HWW}). In other words, property $(X)$ means that elements of $Z$ are 
those elements of $Z^{**}$ which are somehow $\sigma$-additive. If $Z$ has $(X)$, then $Z$ is strongly unique isometric predual for every equivalent norm,
and is weakly sequentially complete (\cite{GT}). Moreover, every space which is $L$-complemented in its bidual has $(X)$ (\cite{Pfitzner}). 
However, the following question seems to be open.
\vskip 3 mm
{\bf Problem}: Assume $n>1$. Does the Banach space $\mathcal F(\mathbb R^n)$ enjoy Property $(X)$ ?
\vs\vs
\section{Closed subspaces of \texorpdfstring{$L^1$}{L1} consisting of continuous functions on a star-shaped domain.}
When $X$ is a nicely placed subspace of $L^1$, a distinguished subspace of $X^*$ is a candidate for being 
the natural predual of $X$. We denote (see Definition IV.3.8 in \cite{HWW}):
$$
X^{\sharp}=\{x^*\in X^*;~x^*~is~\tau_m-continuous~on~B_X\}.
$$
\vskip 3 mm
The following proposition is valid in any separable $L^1$-space, and requests no topology on the measure space.
\vskip 3 mm
\begin{pro}
Let $X$ be a closed subspace of $L^1(m)$. Let $(T_n)$ be a sequence of bounded linear operators from $L^1(m)$
to itself such that $\lim_n\Vert T_n(f) - f\Vert_1=0$ for every $f\in L^1$. We assume that:

\item $(1)$ $T_n(X)\subset X$ for every $n\geq 1$, and the restriction of $T_n$ to $X$ is a weakly compact operator.

\item $(2)$ $T_n$ is $(\tau_m- \tau_m)$-continuous on $\Vert~.~\Vert_1$-bounded subsets of $L^1$ for every $n\geq 1$.

Then the space $X$ is nicely placed and is isometric to the dual $(X^{\sharp})^*$ of the space $X^{\sharp}$.

\end{pro}

\begin{proof} Let $(f_k)$ be a sequence in $B_X$, which $\tau_m$-converges to $g\in L^1$. By $(2)$, for every $n$ the sequence
$(T_n(f_k))_k$ is $\tau_m$-convergent to $T_{n}(g)$. Since by $(1)$ this sequence is weakly relatively compact in $L^1$, we have

$$\lim_k \Vert T_n( f_k - g)\Vert_1=0$$

\noindent and thus $T_n(g)\in X$ for every $n$. But since $(T_n)$ is an approximating sequence it follows that $g\in X$ and thus $X$ is nicely placed.

Note now that for any $h\in L^{\infty}$ and any $n$, the restriction of $T_n^*(h)$ to $X$ is $\tau_m$-continuous on the unit ball of $X$, that is, belongs to $X^{\sharp}$.
If follows that $X^{\sharp}$ separates $X$, and thus by Theorem 1.3 in \cite{MathScand} the space $X^{\sharp}$ is an isometric predual of $X$, and moreover it is
an $M$-ideal in its bidual $X^*$.
\end{proof}
\vs
The following theorem is the main application of Proposition 8. If $\Omega$ is a star-shaped open subset of $\mathbb R^n$, $\rho\in (0, 1)$ and $f$ is any function defined on $\Omega$,
we denote $T_\rho(f)(x)=f(\rho x)$ for every $x\in\Omega$. We equip $\Omega$ with the topology induced by $\mathbb R^n$ and with the Lebesgue measure. We denote by $\tau_K$ the compact-open topology on the space $\mathcal{C}(\Omega)$, that is, the topology of uniform convergence on compact subsets of $\Omega$. With this notation, the following holds:
\vs
\begin{theorem} Let $\Omega$ be a star-shaped bounded open subset of $\mathbb R^n$, and let $X$ be a closed vector subspace of  $L^1(\Omega)$. We assume that $X\subset\mathcal{C}(\Omega)$, and that $T_\rho(X)\subset X$ for every $\rho\in(0, 1)$. Then the closed unit ball $B_X=\{f\in X;~\Vert f\Vert_1\leq 1\}$ of $X$ is $\tau_K$-compact and the topologies $\tau_K$ and 
$\tau_m$ coincide on $B_X$.
\end{theorem}
\begin{proof} 
For showing this, pick any $\rho\in(0, 1)$. We denote by $K=\overline{\Omega}$ the closure of $\Omega$ in $\mathbb {R}^n$, which is compact since $\Omega$ is bounded. The set $T_\rho(B_X)$ is weakly relatively compact in $L^1(\Omega)$, hence $T_{\rho^2}(B_X)$ is weakly relatively compact in $\mathcal{C}(K)$, and thus pointwise (on K) relatively compact in $\mathcal{C}(K)$. Since $T_{\rho^2}(B_X)$ is also weakly relatively compact in $L^1(\Omega)$, by Lebesgue's dominated convergence theorem it is $\Vert~.~\Vert_1$- relatively compact in $L^1(\Omega)$, and therefore $T_{\rho^3}(B_X)$ is $\Vert~.~\Vert_{\infty}$-relatively compact in $\mathcal{C}(K)$. Since $\rho\in(0, 1)$ was arbitrary, it follows that $B_X$ is relatively compact in $\mathcal{C}(\Omega)$ for the compact-open topology $\tau_K$. If we let $T_{(n-1)/n}=T_n$ for convenience, we can apply Proposition 8 with the same notation, and conclude that $B_X$ is $\tau_m$-closed in $L^1$. Note now that any $\tau_K$-convergent sequence in $B_X$ is $\tau_m$-convergent, and it follows that its limit belongs to $B_X$ since $B_X$ is $\tau_m$-closed in $L^1$. Therefore $B_X$ is $\tau_K$-compact. Finally, compactness shows that the topologies $\tau_K$ and $\tau_m$ coincide on $B_X$.
\end{proof}

The motivation for this result is that it implies that the unit ball $B_X$ of $X$ is $\tau_m$-compact locally convex, since $\tau_K$ is locally convex. Such subspaces of $L^1$ have been previously studied in some detail (\cite{GKL}, \cite{GKL2}). It can be shown in particular that, under the mild assumption that they enjoy Grothendieck's approximation property, they yield to a satisfactory unconditional decomposition. The precise statement is given below, in the special case considered in Theorem 9. Note that it has been shown by W. B. Johnson and M. Zippin \cite{JZ} that every quotient of $c_0$ is isomorphic to a subspace of $c_0$, hence $(2)$ below actually improves on $(1)$. Observe that $(2)$ implies that $X$ is arbitrarily close to weak*-closed subspaces of $l^1$.

\begin{cor} Let $\Omega$ be a star-shaped bounded open subset of $\mathbb R^n$, and let $X$ be a closed vector subspace of  $L^1(\Omega)$. We assume that $X\subset\mathcal{C}(\Omega)$, and that $T_\rho(X)\subset X$ for every $\rho\in(0, 1)$.  Then $X=(X^{\sharp})^*$ isometrically, where $X^{\sharp}$ denotes the subspace of $X^*$ consisting of the linear forms which are $\tau_K$-continuous on $B_X$. Moreover:
\vskip 2 mm
\item $(1)$ for any $\epsilon>0$, there exists a subspace $E_\epsilon$ of $c_0$ such that $d_{BM}(X^{\sharp}, E_\epsilon) < 1+\epsilon$. 
\vskip 2 mm
\item $(2)$ If $X$ has Grothendieck's approximation property, then for any $\epsilon>0$ there is a quotient space $Y_\epsilon$ of $c_0$ such that $d_{BM}(X^{\sharp}, Y_\epsilon) < 1+\epsilon$. Moreover there exists a sequence of finite rank operators $(A_i)$ on $X^{\sharp}$ such that  
\vskip 2 mm
$(\textrm{a})$ $\sup_{N, \vert\epsilon_i\vert= 1} \Vert \sum_{i=1}^N \epsilon_i A_i \Vert < 1+\epsilon$.
\vskip 2 mm
$(\textrm{b})$ for every $f\in X$, one has $f=\sum_{i=1}^{\infty} A_i^*(f)$, where the series is norm-convergent.
\end{cor}
\begin{proof} Since our assumptions imply that the unit ball $B_X$ of $X$ is $\tau_m$-compact locally convex, $X=(X^{\sharp})^*$ isometrically and $(1)$ follows from Proposition 2.1 in \cite{GKL}. If $X$ has the approximation property, it actually has the unconditional metric approximation property $(UMAP)$ and moreover its natural predual $X^{\sharp}$ is arbitrarily close to quotients of $c_0$ by Theorem 3.3 in \cite{GKL}.
\par
Since $X=(X^{\sharp})^*$ has $(UMAP)$ and $X^{\sharp}$ is an $M$-ideal in its bidual, and thus in particular a strict $u$-ideal, we may apply Theorem 9.2 in \cite{GKS} which shows in particular that $X^{\sharp}$ has $(UMAP)$. Now Theorem 3.8 in \cite{CK} shows the existence of a sequence $(A_i)$ of finite rank operators satisfying $(a)$ and a weaker version of $(b)$ where norm-convergence is replaced by weak*-convergence. But for every $f\in X$, the series $\sum_{i=1}^{\infty} A_i^*(f)$ is weakly unconditionally convergent, hence norm-convergent since $X$ does not contain $c_0$. This concludes the proof.
\end{proof}
\begin{oss}\rm The proof allows to state some more results. Indeed the unit ball $B_X$ is $\tau_m$-closed in $L^1(\Omega)$ and thus by \cite{G} the quotient space $L^1/X$ is weakly sequentially complete. Moreover, since this unit ball is even $\tau_m$-compact locally convex, the space $X$ satisfies by \cite{Kal} the following extension result: if $X\subset Y$ separable, any continuous linear operator from $X$ to a $\mathcal{C}(K)$-space $Z$ extends to a continuous linear operator from $Y$ to $Z$.
\end{oss}
\noindent
{\bf Examples.} 1) Proposition 8 trivially applies to any reflexive subspace $X$ of $L^1$, by taking $T_n= Id_{L^1}$ for all $n$. Note that in such a space $X$ the $\tau_m$-topology coincide with the norm topology, hence $X^{\sharp}=X^*$. This provides examples of spaces $X$ such that $(1)$ and $(2)$ hold true, but $T_n$ does not induce a compact operator on $X^*$- take any infinite dimensional reflexive space $X$.
\vskip 3 mm
2) The subspace $Har(\Omega)$ of $L^1(\Omega)$ which consists of harmonic functions satisfies the assumptions of Theorem 9. This is also the case for the Bergman space $L^1_a(\Omega)$ of integrable holomorphic functions on $\Omega$ star-shaped open bounded in ${\mathbb C}^n\simeq{\mathbb R}^{2n}$,. Actually, it is clear that many spaces of holomorphic functions on unit balls of ${\mathbb C}^n$ provide examples where Theorem 9 and Corollary 10 apply.
\vskip 3 mm
3) More generally, if $G:\Omega\rightarrow (L^{\infty}, weak^*)$ is a continuous function, the space

$$X_G=\{f\in L^1(\Omega);~f(x)=\int_{\Omega} f(\omega)G(x)(\omega)dm(\omega)~for~all~x\in\Omega\}$$

\noindent is  $\Vert~.~\Vert_1$-closed and consists of continuous functions. It follows that the space $X_{(G_i})=\bigcap_{i\in I} E_{G_i}$ is  $\Vert~.~\Vert_1$-closed and consists of continuous functions for an arbitrary collection of continuous maps $(G_i)$. When this space is moreover stable under the dilation operators $T_\rho$, Theorem 9 applies.
\vskip 3 mm
4) If $\Delta$ denotes the Laplace operator, a function $f$ is called biharmonic if $\Delta^{2}(f)=\Delta\circ\Delta(f)=0$. The space of biharmonic functions on $\Omega$ satisfies the assumptions of Theorem 9.
\vskip 3 mm
5) The sequence $(x_k(t)=2^k(t^{2^k}))_{k\geq 1}$ in $L^1([-1, 1])$ is equivalent to the unit vector basis of $l^1$ (\cite{GuMa}) and thus its closed linear span is contained in $\mathcal{C}((-1, 1))$ and is isomorphic to $l^1$. More generally, let $\Lambda=(\lambda_i)_{i\geq 1}$ be an increasing sequence of positive real numbers such that $\inf_i (\lambda_{i+1} - \lambda_i)>0$ and $\sum_{i\geq 1} \lambda_i^{-1}<+\infty$. The M\"untz space $M_1(\Lambda)$ is the closed linear span of the sequence $(t^{\lambda_i})_{i\geq 1}$ in $L^1([0, 1])$. Then $M_1(\Lambda)$ is contained in $\mathcal{C}([0, 1))$ (see \cite{GLu}) and Theorem 9 and Corollary 10 apply to the space $M_1(\Lambda)$. Note that in  \cite{G2} this result is shown using analyticity of the elements of $M_1(\Lambda)$ on $(0, 1]$ but actually continuity suffices as shown above. Also, the point $0$ does not belong to the interior of the unit interval but the reader will check that this causes no inconvenience in the above proofs. We refer to \cite{GaLef} for precise recent results on the geometry of M\" untz spaces. Let us also mention that M\"untz spaces of functions on the cube $[0, 1]^n$ have been investigated (see \cite{Hellerstein} and subsequent works), and Theorem 9 apply to such spaces as well.
\vskip 3 mm
6) Theorem 9 and Corollary 10 have an $L^p$-version for  $p>1$, and actually this version is rather easier since there is no need to enter the ``{Kalton zone}'' $0\leq p<1$ in this case. Recall that Theorem 4.4 in \cite{KW} states in particular that if $1<p<+\infty$, a subspace $X$ of $L^p$ whose unit ball is $\Vert~.~\Vert_1$-compact is arbitrarily close in Banach-Mazur distance to subspaces of $l^p$. Along the lines of the above proofs, it follows that if $1<p<+\infty$ and $X\subset L^p(\Omega)$ is a closed subspace which consists of continuous functions and such that $T_\rho(X)\subset X$ for every $\rho\in(0, 1)$, then for every $\epsilon>0$, there is a subspace $E_\epsilon\subset l_p$ such that $d_{BM}(X, E_\epsilon)<1+\epsilon$. Note that in this case, there is no need to assume any approximation property. 
\vskip 3 mm
7) It is interesting to compare the dilation operators with the approximation schemes from harmonic analysis. Let $\mathbb{T}$ be the unit circle equipped with the Haar measure, and $(\sigma_n)$ the sequence of Fej\'er kernels. If we let $T_n(f)= f*\sigma_n$, then of course $\lim\Vert T_n(f) - f\Vert_1=0$ for every $f\in L^1(\mathbb{T})$. Any translation invariant subspace $X=L^1_\Lambda(\mathbb{T})$ satisfies $T_n(X)\subset X$, and weak compactness is obvious since the $T_n$'s are finite rank operators. However, condition $(2)$ of Proposition 8 fails. Actually, no non-zero weakly compact operator on $L^1$ satisfies $(2)$, since the existence of $0\not=F\in L^{\infty}$ whose restriction to $B_{L^1}$ is $\tau_m$-continuous would follow and there is no such $F$. We now consider two specific examples of translation-invariant subspaces of $L^1(\mathbb{T})$.
\vskip 3 mm
Let $X= L^1_{\mathbb{N}}(\mathbb{T})=H^1(D)$ be the classical Hardy space on the unit disc, seen as a subspace of $L^1(\mathbb{T})$. Then $X$ is nicely placed and $X=(X^{\sharp})^*$ (\cite{G}) but the unit ball $B_X$ is not $\tau_m$-compact. Actually, every infinite-dimensional translation-invariant subspace $L^1_\Lambda(\mathbb{T})$ contains an isomorphic copy of $l^2$ (since $\Lambda$ contains an infinite Sidon set) and thus fails to have a $\tau_m$-relatively compact unit ball. The topology $\tau_m$ is strictly finer than the weak* topology associated with $X^{\sharp}=VMO$ on $B_X$. However, weak* convergent sequences admit subsequences {\sl whose Cesaro means} are $\tau_m$-convergent to the same limit (Corollary 4.3 in \cite{Riesz}). The operators $T_n$ are $(\tau_m-\tau_m)$- continuous on $B_X$, but not on $X$ (by the argument from Example 3.6(b) in \cite{HWW}).
\vskip 3 mm
Let $\Lambda=\bigcup_{n\geq 1}\{k.2^n;~0<\vert k\vert\leq n\}$ and let $X=L^1_\Lambda(\mathbb{T})$. It follows from the proof of Theorem III.1 in \cite{MathScand} that the restrictions of the operators $T_n$ to $B_X$ are $(\tau_m-\tau_m)$-continuous, but however the space $X$ is not nicely placed. This shows in particular that condition $(2)$ of Proposition 8 cannot be weakened: assuming $(\tau_m-\tau_m)$-continuity of the $(T_n)$'s on bounded subsets of $X$ does not suffice to reach the conclusion.

\vskip 5mm

Let us conclude this note with an open question:
\vskip 3 mm
{\bf Problem}: Let $\Omega$ be a star-shaped open subset of $\mathbb{R}^n$, and let $X\subset L^1(\Omega)$ be a Banach space which satisfies the assumptions of Theorem 9. Assume that $f\in X$ vanishes on a neighbourhood of $0$. Does it follow that $f=0$ ?

\end{document}